\documentclass[12pt, reqno]{amsart}

\usepackage[left=1.1in,top=1.2in,right=1.1in,bottom=1.2in]{geometry}
\usepackage{setspace,kantlipsum}
\usepackage{amsfonts}
\usepackage{amssymb}
\usepackage{amsmath}
\usepackage{amsthm}
\usepackage{latexsym}
\usepackage{graphicx}
\usepackage{xypic}
\usepackage{hyperref}
\usepackage{lineno}
\usepackage{tikz-cd} 
\usepackage{subcaption} 
\usepackage{float}
\usepackage[bottom]{footmisc}
\renewcommand{\thefootnote}{\fnsymbol{footnote}}
\newcommand{\ui}[1]{{\left\vert\kern-0.25ex\left\vert\kern-0.25ex\left\vert #1 
    \right\vert\kern-0.25ex\right\vert\kern-0.25ex\right\vert}}
\makeatletter
\@namedef{subjclassname@2020}{%
  \textup{2020} Mathematics Subject Classification}
\makeatother

\usepackage{doc}
\makeatletter
\def\maketitle{\par
      \begingroup
        \def\thefootnote{\fnsymbol{footnote}}%
      \setcounter{footnote}\z@
      \def\@makefnmark{\hbox to\z@{$\m@th^{\@thefnmark}$\hss}}%
      \long\def\@makefntext##1{\noindent
          \ifnum\c@footnote>\z@\relax
            \hbox to1.8em{\hss$\m@th^{\@thefnmark}$}##1%
          \else
          \hbox to1.8em{\hfill}%
            \parbox{\dimexpr\linewidth-1.8em}{\raggedright ##1}%
          \fi}
      \if@twocolumn\twocolumn[\@maketitle]%
      \else\newpage\global\@topnum\z@\@maketitle\fi
      \thispagestyle{titlepage}\@thanks\endgroup
      \setcounter{footnote}\z@
      \gdef\@date{\today}\gdef\@thanks{}%
      \gdef\@author{}\gdef\@title{}\gdef\@dedicatory{}}

\hypersetup{colorlinks,citecolor=red,pdfstartview=FitH, pdfpagemode=None}

\usepackage{xcolor, soul}

\numberwithin{equation}{section}

\theoremstyle{definition}
\newtheorem{theorem}{Theorem}[section]
\newtheorem{lemma}[theorem]{Lemma}
\newtheorem{proposition}[theorem]{Proposition}

\newtheorem{example}[theorem]{Example}

\newtheorem{remark}[theorem]{Remark}

\renewcommand{\ge}{\geqslant}
\renewcommand{\le}{\leqslant}

\def\ui{\|\hspace{-.25mm} |\,}

\def\B{\mathcal B}
\def\R{\mathbb R}
\def\C{\mathbb C}
\def\Fnn{{\mathbb F}_{n\times n}}
\def\Cnn{{\mathbb C}_{n\times n}}
\def\Rnn{{\mathbb R}_{n\times n}}

\def\F{\mathbb F}

\def\H{\mathbb H}

\def\P{\mathbb P}

\def\S{\mathbb S}

\def\tr{{\rm tr\,}}
\def\rank{\mbox{\rm rank\,}}

\def\Gr{{\mbox{\bf Gr}}}
\def\range{{\mbox{Range}\,}}
\def\Exp{{\rm Exp}}

\begin{document}

\title[Grassmannians]{Geometric Properties and Distance Inequalities on Grassmannians}

\author{Tin-Yau Tam}
\address{Department of Mathematics and Statistics\\ University of Nevada, Reno\\ Reno \\ NV 89557-0084\\ USA}
\email{ttam@unr.edu}

\author{Xiang Xiang Wang}
\address{Department of Mathematics and Statistics\\ University of Nevada, Reno\\ Reno \\ NV 89557-0084\\ USA}
\email{xiangxiangw@unr.edu}

\dedicatory{In memory of Professor Ronald L. Smith (1947-2024) who passed away on October 16, 2024. His contributions to matrix theory and his dedication to the academic community will always be remembered.}

%


\begin{abstract} In this paper we obtain inequalities for the geometric mean of  elements in the Grassmannians. These inequalities reflect the elliptic geometry of the Grassmannians as Riemannian manifolds. These include Semi-Parallelogram Law, Law of Cosines and geodesic triangle inequalities.
\end{abstract}

\maketitle


\section{Introduction}

Let $\F$ denote the real field $\R$ or the complex field $\C$.
Let $\Fnn$ be the space of all $n \times n$ matrices over $\F$. Let $\P_n$  the set of $n\times n$ positive definite matrices in $\Cnn$. The geometric mean of $A, B\in \P_n$ is defined as:
\begin{equation}\label{gm}
A \sharp B := A^{1/2}(A^{-1/2}BA^{-1/2})^{1/2}A^{1/2},
\end{equation}
first introduced by Pusz and Woronowicz \cite {PW75} in 1975 and has since been extensively studied. 
This geometric mean has a natural geometric interpretation, as $\P_n$ is a Riemannian manifold \cite {He78} and the geometric mean $A \sharp B$  represents the mid-point of the unique geodesic joining $A$ and $B$ \cite{RB07, LLT14}. 
The Riemannian structure of $\P_n$ is given by
\[
Q_p(X, Y) = 
   \tr (p^{-1}Xp^{-1}Y),
\]
where $X, Y \in T_p\P_n$, the tangent space of $\P_n$ at $p\in \P_n$.
 This structure turns $\P_n$ into a metric space, where the distance between $A, B\in \P_n$ is
\begin{equation}\label{d_P}
d_{\P_n}(A,B) :=  \Big(\sum_{i=1}^n \log ^2\lambda_i(BA^{-1})\Big)^{1/2}.
\end{equation}
Here  $\lambda_1, \cdots \lambda_n\in \R$ are the eigenvalues of $BA^{-1}$. See \cite{GLT21} for some recent developments on the geometric mean.

In this paper, we consider $t$-geometric mean for Grassmannians $\Gr_{n,k}(\F)$, where $t\in[0,1]$. Grassmannians have nonnegative sectional curvature, while $\P_n$ has nonpositive  sectional curvature. 
Grassmannians  are fundamental in various branches of mathematics, including differential geometry, algebraic geometry, and theoretical physics. Their geometric properties are not only of theoretical interest but also have practical applications, such as in  image recognition \cite{Hu14, Liu19, Souza23}.
As a model for the Grassmannians $\Gr_{n,k}(\C)$, we use orthogonal projection matrices of rank $k$:
$$
\Gr_{n,k}(\C)=\{P\in \Cnn: P^2 =P, \, P^*=P,\ \rank P = k\} \subset \H_n,
$$
where $\H_n$ is the  space of $n\times n$ Hermitian matrices. Its real analog is 
$$
\Gr_{n,k}(\R)=\{P\in \Rnn: P^2 =P, \, P^\top=P,\ \rank P = k\} \subset \S_n,
$$
where $\S_n$ is the space of $n\times n$ real symmetric matrices. 
See \cite {AMS04, YWL22} for more on the manifold structure and properties of Grassmannians. 

We would like to remark, as noted in \cite [Theorem 10]{W67}, that if  $k=1$ or $k=n-1$,  there is a unique  geodesic   between any two points in $\Gr_{n,k}(\F)$. However, when $k\ge 2$ or $n\ge k+2$, there are either countably many or uncountably many geodesics passing through any two points  in $\Gr_{n,k}(\F)$. This contrasts sharply with $\P_n$ which has a unique geodesic for any pair of points $P,Q\in\P_n$ \cite {LLT14}.

 The injectivity radius $R_P$ at $P \in \Gr_{n,k}(\F)$ is the biggest $R>0$ 
 such that the exponential map from a neighborhood of $0\in T_P\Gr_{n,k}(\F)$ to a neighborhood of $P$ is a diffeomorphism. It is positive since the exponential map is a local diffeomorphism at each point $P$ in the compact  $\Gr_{n,k}(\F)$.
The {\it injectivity radius} $r_\F$ of $\Gr_{n,k}(\F)$
 is defined as:
$$
r_\F:=\min \{R_P: P\in \Gr_{n,k}(\F) \} >0.
$$
We define
$$\rho_\F := \min \left\{ \frac {r_\F} 2,\ \frac \pi {2 \delta_{\F} ^{1/2}}  \right\},
$$
where $\delta_{\F}$ is the maximum of the sectional curvatures $\kappa$ of $\Gr_{n,r}(\F)$. 
It is known that every metric ball $\B$ of radius less than $\rho_\F$  is {\it strongly convex}
\cite{K70, GK73}, meaning that for any points  $P, Q\in \B$ there is a unique minimizing geodesic \cite {Cheeger75}, i.e., the geodesic that represents the shortest path, from $P$ to $Q$ in $\Gr_{n,}(\F)$ and this geodesic  lies entirely within $\B$.  According to Wong \cite[Theorems 1-3]{W68a},
$$
\delta_\R \le 2, \quad \delta_\C \le 4.
$$
Thus we have
\[
0< \min \Big \{ \frac {r_\R}2,\ \frac \pi {2 \sqrt 2} \Big \} \le \rho_\R, \quad 0<  \min\Big  \{ \frac {r_\C}2,\ \frac \pi 4 \Big  \} \le\rho_\C. 
\]
In this paper we focus on  the local property of Grassmannians, especially  studying  behaviors within a metric ball $\B$ of radius less than $\rho_\F$ because this ensures the existence of a unique geodesic joining any two points within $\B$.

Let $P, Q\in  \Gr_{n,k}(\F)$ lie within a metric ball $\B$ of radius less than $\rho_\F$. The exponential map  \cite [p.30-32]{He78} at $P$, denoted by $\Exp_P$, is a diffeomorphism from a neighborhood of $0\in T_P\Gr_{n,k}(\F)$ to a neighborhood of $P$. 
 This enables us to define via $\Exp_P$ the tangent vector $\overrightarrow {PQ}$ for the geodesic $\gamma$ joining $P$ and $Q$:
\[
\overrightarrow {PQ}  = \Exp^{-1}_PQ \in T_P\Gr_{n,k}(\F)\quad \mbox{i.e.,}\quad  \Exp_P \overrightarrow {PQ} = Q.
\]
The geodesic $\gamma(t)$  can be expressed as 
\begin{equation}\label{geo-abs}
\gamma(t) = \Exp_P(t\, \Exp^{-1}_PQ) = \Exp_P(t\, \overrightarrow {PQ}), \quad t\in[0,1].
\end{equation} 
The distance between $P$ and any point $R$  lying on $\gamma$ is given by  \cite{MJ63, KH77, MF06}:

\begin{equation}\label{dist-abs}
d(P,R) = \langle \Exp_P^{-1} R, \Exp_P^{-1}R\rangle_P^{1/2},
\end{equation}
where  $\langle \cdot, \cdot \rangle_P$   is the Riemannian inner product at $P$. 

We define the $t$-geometric mean $P\sharp_t Q$ as
\[
 \gamma(t) = P\sharp_tQ, \quad t\in[0,1], \quad \mbox{and}\ P\sharp Q := P\sharp_{\frac 12}Q.
 \]
Under the condition that $(I-2P)(I-2Q)$ has no negative eigenvalues, the $t$-geometric mean $P\sharp_tQ $  can be  explicitly expressed, as described  by 
Batzies, H\"{u}per, Machado, and Leite \cite [p.91]{BK2015},   which we will outline below.
The matrix $r_P$  
$$r_P:=2P-I= P - (I-P),$$ 
is  the reflection through $(\range P)^\perp$, where $I-P$ is the orthogonal projection onto $(\range P)^\perp$, the orthogonal complement of  $\range P$, the range of the matrix $P$, and thus
\begin{equation}\label{reflection}
(I-2P)^2 = I.
\end{equation}
When the unitary (orthogonal for the real case) matrix $(I-2P)(I-2Q)$ has no negative eigenvalues,
\begin{equation}\label{geodesic}
P\sharp_tQ  := \gamma(t) = e^{t\Omega} Pe^{-t\Omega},
\end{equation}
where
\begin{equation}\label{Omega}
\Omega := \frac 12 \log [(I-2Q)(I-2P)]
\end{equation}
and $\overrightarrow {PQ}$ can be expressed as
\begin{equation}\label{Omega,P}
\overrightarrow {PQ}=[P, \Omega]:=P\Omega- \Omega P.
\end{equation}
Hence
 \begin{equation}\label{geodesic1}
\gamma(t) = [(I-2Q)(I-2P)]^{t/2} P [(I-2Q)(I-2P)]^{-t/2}.
\end{equation}
Note that 
\begin{equation}\label{-Omega}
\log [(I-2Q)(I-2P)]=-\log [(I-2P)(I-2Q)]
\end{equation}
 since $ (I-2Q)(I-2P)=[(I-2P)(I-2Q)]^{-1}.$
 Also, from \cite [p.86]{BK2015}, we have 
\begin{equation}  \label{omega}
[P,\Omega]=(I-2P)\Omega=-\Omega (I-2P).
\end{equation}
The distance between $P$ and $Q$ is  given by 
\begin{eqnarray}\label{dist}
d(P,Q) &=& \left[-\frac 14\tr (\log ^2 (I-2Q)(I-2P))\right]^{1/2} \nonumber\\
& =&\frac{1}{4}\|\log(I-2Q)(I-2P)\|_F = \|\Omega\|_F=\|\overrightarrow {PQ}\|_F,
\end{eqnarray}
where $\|\cdot\|_F$ is the Frobenius norm. 

\begin{remark}
  The expressions for the geodesic \eqref{geodesic} and the distance \eqref{dist} hold only when the product \((I-2P)(I-2Q)\) has no negative eigenvalues. Under these conditions, the tangent vector \(\overrightarrow{PQ}\) and the distance can be expressed in terms of \(\Omega\) as shown.
\end{remark}

The structure of this paper is as follows. In section 2, we establish the Semi-Parallelogram Law and the Law of Cosines for  Grassmannians.  Additionally, we prove that  for a geodesic triangle with vertices $A$, $B$, and $C$, the inequality $d(A \sharp B, A \sharp C) \geq \frac{1}{2} d(A, B)$ holds.
In section 3, we extend this result, showing that  for any $t \in [0, 1]$, the inequality $d(A \sharp_t B, A \sharp_t C) \geq t d(A, B)$ holds.
Examples are given to illustrate that these are local properties, meaning they describe behaviors within a metric ball $\B$ of radius less than $\rho_\F$.

\section{Distance inequalities on geometric means}

If $t\in[0,1]$ is interpreted as time, then $\gamma(t)$ is the velocity function. The proposition below asserts that $\gamma$ has constant speed, or equivalently, 
the function $$[0,1]\to [0,d(P,Q)], \quad t\mapsto d(P, P\sharp_tQ)$$  is linear.
\begin{proposition}\label{linear}
Let $P, Q\in  \Gr_{n,k}(\F)$ lie within a metric ball $\B$ of radius less than $\rho_\F$. Then for any $t\in [0,1]$,
\begin{equation}\label{dt=td}
 d(P, P\sharp_tQ)=t\, d(P,Q).
\end{equation}
\end{proposition}

\begin{proof} 
By \eqref{geo-abs}, the unique geodesic  joining $P$ and $Q$ is given by $\gamma(t) = \Exp_P(t\, \Exp^{-1}_PQ)$, $t\in[0,1]$. Using \eqref{dist-abs}, we have
\begin{eqnarray*}
d(P, P\sharp_tQ) &=& d(P, \gamma(t)) \\
&=& 
\langle \Exp_P^{-1} ( \Exp_P(t\, \exp^{-1}_PQ)), \Exp_P^{-1}( \Exp_P(t\, \Exp^{-1}_PQ))\rangle_P^{1/2} \\
&=& 
\langle t\, \Exp^{-1} _PQ, t\Exp^{-1}_PQ\rangle_P^{1/2} \\
&=& t\langle  \Exp^{-1}_PQ,  \Exp^{-1}_PQ\rangle_P^{1/2} \\
&=& t d(P,Q).
\end{eqnarray*}
\end{proof}
Proposition \ref{linear} has some nice implications, namely,
\begin{equation}\label{1-t}
P\sharp _t Q = Q\sharp_{1-t}P, \quad t\in (0,1),
\end{equation}
\begin{equation}\label{st}
P \sharp_s (P\sharp_{t} Q) = P\sharp _{st}Q, \quad  s, t\in (0,1),
\end{equation}
\begin{equation}\label{s+t}
(P \sharp_s Q)\sharp_{\frac t{1-s}} Q = P\sharp _{s+t}Q, \quad  s, t, s+t\in (0,1).
\end{equation}
To derive \eqref{1-t}, Proposition \ref{linear} implies 
\[
d(Q, Q\sharp_{1-t}P) = (1-t)d(Q,P) = (1-t)d(P,Q)
\] and
\[
d(P, Q\sharp_{1-t} P) + d(Q, Q\sharp_{1-t}P) = d(P,Q),
\]
 and thus $d(P, Q\sharp_{1-t}P) = td(P,Q)=d(P,P\sharp_tQ)$. This implies $P\sharp _t Q = Q\sharp_{1-t}P$ since $\gamma(t$) is the unique geodesic joining $P$ and $Q$.
To prove \eqref{st},  Proposition \ref{linear} implies
\[
d(P,  P\sharp_s (P\sharp_{t} Q))= s\, d(P, P\sharp_{t} Q)=st\, d(P,Q) = d(P, P\sharp_{st}Q).
\]
We have $P \sharp_s (P\sharp_{t} Q) = P\sharp _{st}Q$ since $\gamma(t)$ is the unique geodesic joining $P$ and $Q$. Similarly, \eqref{s+t} can be derived.

\begin{remark} 
We would like to point out that  (\ref{dt=td}) holds {\it globally} for any two points in $\P_n$  \cite[Theorem 6.1.6] {RB07}, 
and more generally for symmetric spaces of noncompact type.
The Semi-Parallelogram Law for $\P_n$ \cite[Theorem 6.1.9] {RB07} reflects the hyperbolic geometry of $\P_n$ which is a symmetric space of noncompact type. 
For more details on the equivalence between the Semi-Parallelogram Law and the nonpositive curvature of the underlying Riemannian manifold, see \cite [p.309-311, Theorem 3.5 in Chapter XI, and Chapter XII] {L99}.
\end{remark}

We will establish the {\it local} Semi-Parallelogram Law for $\Gr_{n,k}(\F)$, which reflects the elliptic geometry of the Grassmannians, characterized by its nonnegative sectional curvature \cite {W68a}.  This geometry, primarily driven by the manifold's curvature, differs from Euclidean geometry, where straight lines are geodesics, the Parallelogram  Law holds, and the sum of the interior angles of a triangle is $\pi$. 

\begin{theorem}\label{laws} Let $A, B, C\in  \Gr_{n,k}(\F)$ be  three points  within a metric ball $\B$ of radius less than $\rho_\F$.
\begin{enumerate}
\item (Semi-Parallelogram Law)
\label{SemiPa}
 If $M$ is  the midpoint of the geodesic from $A$ to $B$, then

\begin{equation}\label{semipara}
d^2(M,C)\geq \frac{d^2(A,C)+d^2(B,C)}{2}-\frac{d^2(A,B)}{4}.
\end{equation}

\item (Law of Cosines)  Suppose that $A, B, C$ are distinct.  
Let $\alpha, \beta, \gamma$ be
the angles 
at $A, B, C$, respectively.
Then
\begin{equation}\label{cosine}
d^2(B,A)\leq d^2(C,A)+d^2(C,B)-2d(C,A)d(C,B)\cos \gamma.
\end{equation}
and 
\begin{equation}\label{cosine2}
d(B,A)\geq d(C,A)\cos \alpha+d(C,B)\cos \beta.
\end{equation}
Moreover, 

\begin{equation}\label{cosine3}
\alpha+ \beta+ \gamma \ge \pi.
\end{equation}
\end{enumerate}
\vspace{-1cm}
\begin{figure}[H]
\begin{center}
\includegraphics[width=0.35\textwidth, angle=0]{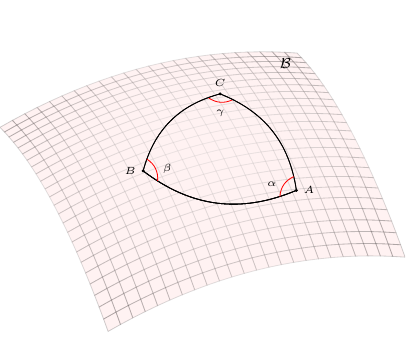}
\end{center}
\caption{Law of Cosines}
\end{figure}
\end{theorem}

\begin{proof} 
The sectional curvature at each point $P\in \B\subset \Gr_{n,k}(\F)$ is nonnegative and $\B$ is convex. We note that the terminology used in \cite[p. 220]{L99} refers to this property simply as convex, rather than strongly convex, and we reserve the term convex for function convexity in Examples \ref{3.5} and \ref{3.6}.  As shown in
 \cite [p.308]{L99},  the exponential map at $P$, $\Exp_P: V_P\to \B$, is an isomorphism for a neighborhood $V_P$ of $0$ of the tangent space $T_P\B$, making  $\Exp_P$ metric non-increasing from $V_P$ to $\B$. Specifically, this means:
\begin{equation}\label{non-inc}
\|(d\, \Exp_P)_v(w)\|\le \|w\|\quad \mbox{for all}\ v\in V_P,\ w\in T_vV_P,
\end{equation}
where $d\, \Exp_P$ denotes the differential of the exponential map $\Exp_P$.
Note that $\Exp_P$ is
 metric preserving along rays from the origin \cite [p.305]{L99}:
 \[
 \|(d\Exp_P)v(w)\|=\|w\|,\quad \mbox{for all}\ v\in V_P, w\in T_vV_P\ \mbox{and}\ w\ \mbox{is a scalar multiple of}\ v.
 \]
   
 Here $\|\cdot\|$ is the norm induced by the Riemannian metric $\langle \cdot, \cdot \rangle_P$.
By an argument similar to that in Remark 2.2 and Lemma 3.1 of \cite{Conde2009} (for the case of the manifold with nonpositive curvature), we can analogously conclude the following in Grassmannians.
 If $\sigma:[0,1]\to \B$ is a smooth curve,
then 
\[L(\Exp^{-1}_P \circ \sigma)\geq L(\sigma),
\]
where $L(\cdot)$ is the arc length of the curve \cite [p.189]{L99}.
Therefore,
\[\|x-y\|\geq d(\Exp_P x , \Exp_P y) \]
for any $x,y\in T_vV_P$.

(1) As $M$ is the midpoint of the geodesic from $A$ to $B$, there is  $v_1 \in T_M \B$ such that 
$$
A = \Exp_M v_1, \quad B=\Exp_M v_2, \quad M = \Exp_M 0,
$$
where $v_2=-v_1$.
For $C = \Exp_M v$, where $v \in T_M \B$, the Parallelogram Law for  the Euclidean space $T_M\B$ gives
$$
\delta^2(v_1, v_2)+4\delta^2 (v,0) = 2\delta^2 (v,v_1)+ 2\delta^2 (v,v_2),
$$
where $\delta(x,y) = \|x-y\|:= \langle x-y,x-y\rangle_M^{1/2}$ for $x, y\in T_M\B$. Under the exponential map $\Exp_M$, the distances on the left side are preserved, but the distances on the right side are not increased. So we have
$$
d^2(A, B)+4d^2(C,M) \geq 2d^2(C, A)+ 2d^2(C, B),
$$
which is the Semi-Parallelogram Law for  $\B\subset \Gr_{n,k}(\F)$.

 (2) Similarly, there are $v_1, v_2\in T_C \B$  such that 
\[
A = \Exp_Cv_1, \quad B=\Exp_Cv_2, \quad C=\Exp_C 0.
\]
The Law of Cosines for the Euclidean space $T_C\B$ states:
\[
\delta^2(v_2,v_1) =  \delta^2(0,v_1)+\delta^2(0,v_2)-2\delta(0,v_1)d(0,v_2)\cos \gamma.
\]
 Under the exponential map $\Exp_C$, the distances on the right side are preserved but the distance on the left side is not increased. So
\[
d^2(B,A)\leq d^2(C,A)+d^2(C,B)-2d(C,A)d(C,B)\cos \gamma
\]
i.e., we proves \eqref{cosine}.
By symmetry, we also have
\[
d^2(C,B) \leq d^2(C,A)+d^2(B,A) -2d(C,A) d(B,A) \cos \alpha,
\]
and
\[ 
d^2(C,A)  \leq d^2(C,B)+d^2(B,A) - 2d(B,A)d(C,B)\cos \beta . 
\]
Since $A$ and $B$ are distinct, $d(B,A)\not=0$ and thus we have
\[
d(C,A)\cos \alpha \leq \frac{d^2(C,A)+d^2(B,A)-d^2(C,B)}{2d(B,A)}
\]
and 
\[d(C,B)\cos \beta \leq \frac{d^2(C,B)+d^2(B,A)-d^2(C,A)}{2d(B,A)}. \]
So we conclude
\[ 
d(C,A)\cos \alpha+d(C,B)\cos \beta\leq d(B,A),\]
i.e., \eqref{cosine2} is established.

%
Finally, we can find a plane triangle with sides of lengths $d(A, B)$, $d(B,C)$, $d(A,C)$ and corresponding angles $\alpha', \beta', \gamma'$. 
The  Law of Cosines implies that $\alpha'\le \alpha$, $\beta'\le \beta$ and $\gamma'\le \gamma$, so we have $\alpha+ \beta+ \gamma \ge \pi$, which is \eqref{cosine3}.
 \end{proof}

In  (\ref{SemiPa}) of Theorem \ref{laws}, if  $D$ in $\B$ exists such that the $M$ is  the midpoint of the geodesic from $C$ to $D$.
Then we have 
\[ d^2(M,C)\geq \frac{d^2(A,C)+d^2(B,C)}{2}-\frac{d^2(A,B)}{4},\]
 
\[d^2(M,D)\geq \frac{d^2(A,D)+d^2(B,D)}{2}-\frac{d^2(A,B)}{4},\]
and \[ d^2(M,C)=d^2(M,D)=\frac{d^2(C,D)}{4}.\]
Hence we obtain
\[ 
\begin{split}
d^2(M,C) + d^2(M,D)&\geq \frac{d^2(A,C)+d^2(B,C)+d^2(A,D)+d^2(B,D)}{2}-\frac{d^2(A,B)}{2}\\
\frac{d^2(C,D)}{2}&\geq   \frac{d^2(A,C)+d^2(B,C)+d^2(A,D)+d^2(B,D)}{2}-\frac{d^2(A,B)}{2}.\\
\end{split}
\]
Finally, we have 
\begin{equation}
d^2(C,D)+d^2(A,B) \geq d^2(A,C)+d^2(B,C)+d^2(A,D)+d^2(B,D).
\end{equation}
\vspace{-1.5cm}
\begin{figure}[H]
\begin{center}
\includegraphics[width=0.35\textwidth, angle=0]{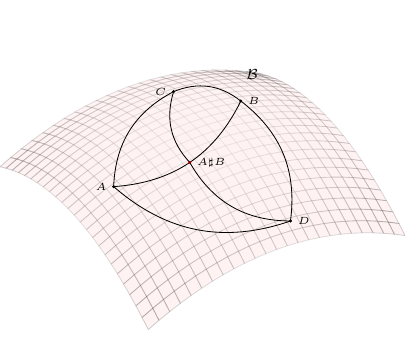}
\end{center}
\caption{$
d^2(C,D)+d^2(A,B) \geq d^2(A,C)+d^2(B,C)+d^2(A,D)+d^2(B,D)$}
\end{figure}

\begin{remark} 
Theorem \ref{laws} (1) asserts that  for any $A, B\in \B$, there is  $X\in\B$ such that for all $C\in \B$, 
\begin{equation}\label{X}
d^2(X,C)\geq \frac{d^2(A,C)+d^2(B,C)}{2}-\frac{d^2(A,B)}{4},
\end{equation}
and the midpoint $X=M$ between $A$ and $B$ satisfies this inequality.
Moreover,
$X$  in \eqref{X} is unique. Let $M\in \B$ is the midpoint of $A$ and $B$. Then
$$d^2(M,C)\geq \frac{d^2(A,C)+d^2(B,C)}{2}-\frac{d^2(A,B)}{4}$$ for any $C\in \B$.
Suppose that there is a  $X\in \B$ with $X\neq M$ such that 
\[
d^2(X,C)\geq \frac{d^2(A,C)+d^2(B,C)}{2}-\frac{d^2(A,B)}{4}
\] for any $C\in \B$. Define the sets
\[
S_1=\{C\in \B: d(X,C)\geq d(M,C)\}, \quad S_2=\{C\in \B: d(X,C)< d(M,C)\}.
\]
Then $S_1\cap S_2=\emptyset$ and $S_1\cup S_2=\B$.
Since $X\in S_2$, we have
\[d^2(M,X)> 0 = d^2(X,X)\geq \frac{d^2(A,X)+d^2(B,X)}{2}-\frac{d^2(A,B)}{4}.\]
Using the triangle inequality $d(A,B)<d(A,X)+d(B,X)$, we obtain
\[
\begin{split}0&\geq \frac{d^2(A,X)+d^2(B,X)}{2}-\frac{d^2(A,B)}{4}\\
&\geq \frac{d^2(A,X)+d^2(B,X)}{2}-\frac{(d(A,X)+d(B,X))^2}{4}\\
&\geq \frac{(d(A,X)-d(B,X))^2}{4}.
\end{split}
\]
Thus, all inequalities must hold as equalities. From the last equality, we deduce $d(A,X)=d(B,X)$, and by the first equality, these distances are equal to $\frac 12 d(A,B)$.
Since there is only one shortest path between $A$ and $B$ in $\B$, the geodesics  between $A$ and $B$ via $M$ must be identical to the geodesic via $X$. Therefore, $X=M$, 
which contradicts our assumption that $X\not=M$.
Hence we conclude that $X$  in \eqref{X} is unique, and is the midpoint of $A$ and $B$.
\end{remark}

\begin{remark} \label{remark} 
The Semi-Parallelogram Law \eqref{semipara}  fails to hold if $A,B,$ and $C$ lie outside the ball $\B$, even if the pairwise products of 
$(I-2A)$, $(I-2B)$, and $(I-2C)$ have no negative eigenvalues. This is illustrated by the following numerical example.
\end{remark}

\begin{example}
Consider the following three matrices in $\Gr_{2,2}(\C)$:
\[A=\begin{pmatrix} 0.5709 + 0.0000i	&-0.1226 -0.4795i\\
-0.1226 + 0.4795i	 & 0.4291 + 0.0000i \end{pmatrix} \]
\[ B=\begin{pmatrix}  0.1673 + 0.0000i	 &0.3063 - 0.2132i\\
0.3063 + 0.2132i	&0.8327 + 0.0000i \end{pmatrix}\]
\[C=\begin{pmatrix}   0.9638 + 0.0000i	&0.0316 + 0.1842i\\
0.0316 - 0.1842i	&0.0362 + 0.0000i\end{pmatrix}.\]
The products $(I-2A)(I-2B)$, $(I-2A)(I-2C)$, $(I-2B)(I-2C)$ have no
negative eigenvalues.
By \eqref{geodesic1}, the midpoint of $A$ and $B$ is
\[M=\begin{pmatrix} 0.7798 - 0.0000i	&-0.3075 - 0.2778i\\
-0.3075 + 0.2778i	& 0.2202 + 0.0000i \end{pmatrix}.  \]
From \eqref{dist},  the distances are $d(A,C)= 0.6401$,
$d(B,C)=0.8476 >\frac 14\pi$, $d(A,B)=0.4970$ and $d(M,C)=0.4567$,
and $d^2(M,C)- [\frac{d^2(A,C)+d^2(B,C)}{2}-\frac{d^2(A,B)}{4}]=-0.2937$.
Then $$d^2(M,C)\leq \frac{d^2(A,C)+d^2(B,C)}{2}-\frac{d^2(A,B)}{4}.
 $$
 i.e., \eqref{semipara} fails to hold. This example also implies that $ \frac 12 {r_\C} \le \rho_\C<\frac 14 \pi$ for $\Gr_{2,2}(\C)$.
 \end{example}

By applying the Semi-Parallelogram Law, we can deduce the following geometric property for a geodesic triangle.

\begin{proposition}\label{half}
Let $A, B, C\in  \Gr_{n,k}(\F)$ lie within a metric ball $\B$ of radius less than $\rho_\F$. Then
\begin{equation} \label{midineq}
d(A\sharp B,A\sharp C)\geq \frac{d(B,C)}{2}.
\end{equation}
\vspace{-1cm}
\begin{figure}[H]
\begin{center}
\includegraphics[width=0.40\textwidth, angle=0]{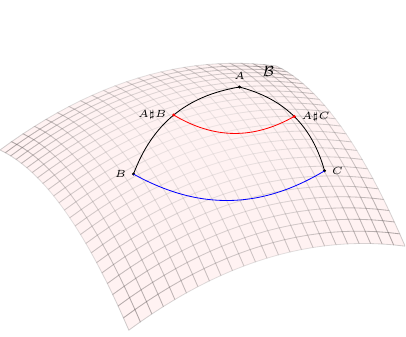}
\end{center}
\caption{$d(A\sharp B,A\sharp C)\geq \frac 12 d(B,C)$}
\end{figure}
\end{proposition}

\begin{proof} 
Let $M_1=A\sharp B$ and $M_2=A\sharp C$.
Applying the Semi-Parallelogram Law (\ref{semipara}) to the geodesic triangle $\Delta(A,B,C)$ with vertices $A, B, C$, and $M_1$ being the midpoint of $A$ and $B$, we obtain
\begin{equation}\label{first}
d^2(M_1,C)\geq \frac{d^2(A,C)+d^2(B,C)}{2}-\frac{d^2(A,B)}{4}.
\end{equation}
Next, applying (\ref{semipara}) again to the geodesic triangle $\Delta(A,C, M_1)$,  with $M_2$  the midpoint of $A$ and $C$, gives\begin{equation}\label{second}
d^2(M_1,M_2)\geq \frac{d^2(M_1,C)+d^2(M_1,A)}{2}-\frac{d^2(A,C)}{4}.
\end{equation}
Substituting (\ref{first}) into (\ref{second}), we get
\[d^2(M_1,M_2)\geq \frac{d^2(A,C)+d^2(B,C)}{4}-\frac{d^2(A,B)}{8}+\frac{d^2(M_1,A)}{2}-\frac{d^2(A,C)}{4}.\]
Since $d(M_1,A)=\frac 12 d(A,B)$, this simplifies to
$d^2(M_1,M_2)\geq\frac 14{d^2(B,C)}.$
Thus we conclude that $d(M_1,M_2)\geq \frac 12 d(B,C)$, which is \eqref{midineq}.
\end{proof}

\section{Generalized Distance Inequality $d(A\sharp_t B, A\sharp_t C) \geq   t d(B,C)$ for $t\in [0,1].$}
It follows from \eqref{midineq} that $d(A\sharp_t B, A\sharp_t C) \geq  t d(B,C)$ for any  $t= \frac 1{2^k}$, where $k$ is a positive integer.
It is natural to ask whether \eqref{midineq}  can be extended   to $t\in[0,1]$, that is,
 \[
d(A\sharp_t B, A\sharp_t C) \geq   t d(B,C), \quad t\in [0,1].
\]
The answer is affirmative. To prove it, we first establish the following  lemma.
 
\begin{lemma}\label{forp}
Let $A, B, C\in  \Gr_{n,k}(\F)$ lie within a metric ball $\B$ of radius less than $\rho_\F$. We have
\begin{equation}\label{anyp}
 \frac{p-1}{p} d(B,C) \leq d(A\sharp_{\frac{p-1}{p}} B, A\sharp_{\frac{p-1}{p}} C) , \quad p =1, 2, 3, \dots
\end{equation}
 
\end{lemma}

\begin{proof}  We will apply the Semi-Parallelogram Law \eqref{semipara}  strategically to various geodesic triangles which are triangles within the geodesic triangle $\Delta(A,B,C)$ with vertices $A, B, C$. This approach is illustrated by the following picture when $p=3$.
\vspace{-1.5cm}
\begin{figure}[H]
\begin{center}
\includegraphics[width=0.5\textwidth, angle=0]{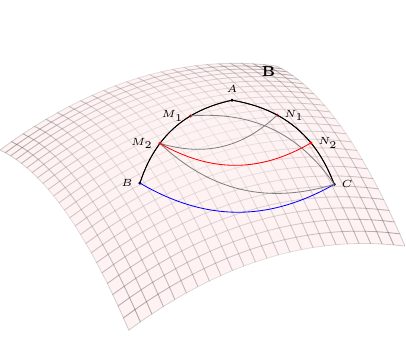}
\end{center}
\caption{Divisions for $p=3$.}
\end{figure}
To proceed, we equally subdivide the paths $AB$ and $AC$ into $p$ segments each:
\[
M_i=A\sharp_{\frac{i}{p}} B,\quad N_i=A\sharp_{\frac{i}{p}} C,\quad i=0,1, \dots ,p,
\] 
where $M_0=A$, $M_p=B$, $N_0=A$, and 
$N_p=C$.

For $j=1,\dots ,p-1$,  applying  \eqref{semipara} to the geodesic triangle $\Delta(N_{p-(j+1)}, N_{p-(j-1)},M_{p-1})$ with  $N_{p-j}$ being the midpoint of $N_{p-(j+1)}$ and $N_{p-(j-1)}$,  we have
 \begin{eqnarray}\label{M1}
 d^2(M_{p-1}, N_{p-j})&\geq& \frac{d^2(M_{p-1}, N_{p-(j+1)})+d^2(M_{p-1}, N_{p-(j-1)})}{2}- \frac{d^2(N_{p-(j+1)},N_{p-(j-1)})}{4} \nonumber\\
 &\geq&  \frac{d^2(M_{p-1}, N_{p-(j+1)})+d^2(M_{p-1}, N_{p-(j-1)})}{2}-\frac{d^2(A,C)}{p^2},
 \end{eqnarray}
because  $d(N_{p-(j+1)},N_{p-(j-1)})= \frac 2p d(A,C)$. Similarly, applying 
 \eqref{semipara} to the geodesic triangle 
$\Delta(M_{p-(j+1)}, M_{p-(j-1)}, N_p)$ with  $M_{p-j}$ being the midpoint of $M_{p-(j+1)}$ and  $M_{p-(j-1)}$,  we obtain

 \begin{equation}\label{N1}
 d^2(M_{p-j}, N_{p})\geq  \frac{d^2(M_{p-(j+1)}, N_{p})+ d^2(M_{p-(j-1)}, N_{p})}{2}-\frac{d^2(A,B)}{p^2},
 \end{equation}
where $j=1,\dots ,p-1$.
 Setting $j=p-1, \dots, 2, 1$ in \eqref{M1} produces the following $p-1$ inequalities:

 
 \begin{subequations}
\begin{align} 
 &d^2(M_{p-1}, N_{1})\geq  \frac{d^2(M_{p-1}, N_{0})+d^2(M_{p-1}, N_{2})}{2}-\frac{d^2(A,C)}{p^2} \label{1}\\
&d^2(M_{p-1}, N_{2})\geq  \frac{d^2(M_{p-1}, N_{1})+d^2(M_{p-1}, N_{3})}{2}-\frac{d^2(A,C)}{p^2} \label{2}\\
&d^2(M_{p-1}, N_{3})\geq  \frac{d^2(M_{p-1}, N_{2})+d^2(M_{p-1}, N_{4})}{2}-\frac{d^2(A,C)}{p^2}\label{3}\\
&\qquad \qquad \vdots  \nonumber\\
&d^2(M_{p-1}, N_{i+1})\geq  \frac{d^2(M_{p-1}, N_{i})+d^2(M_{p-1}, N_{i+2})}{2}-\frac{d^2(A,C)}{p^2}. \label{4}\\
&\qquad \qquad  \vdots  \nonumber\\
&d^2(M_{p-1}, N_{p-1})\geq  \frac{d^2(M_{p-1}, N_{p-2})+d^2(M_{p-1}, N_{p})}{2}-\frac{d^2(A,C)}{p^2}.  \label{MN1}
  \end{align}
\end{subequations} 
Substituting \eqref{1} into \eqref{2} gives
\[
\begin{split}
d^2(M_{p-1}, N_{2}) & \geq  \frac{1}{2}d^2(M_{p-1}, N_{1})+\frac{1}{2}d^2(M_{p-1}, N_{3})-\frac{d^2(A,C)}{p^2}\\
&\geq  \frac{1}{2}\left[\frac{d^2(M_{p-1}, N_{0})+d^2(M_{p-1}, N_{2})}{2}-\frac{d^2(A,C)}{p^2} \right]
+\frac{1}{2}d^2(M_{p-1}, N_{3})-\frac{d^2(A,C)}{p^2}\\
&= \frac{1}{4} d^2(M_{p-1}, N_0)+\frac{1}{4}d^2(M_{p-1}, N_{2})  +\frac{1}{2}d^2(M_{p-1}, N_{3})-\frac{3}{2p^2} d^{2}(A,C),
\end{split}
 \]
which simplifies to
\[ \frac{3}{4} d^2(M_{p-1},N_2)\geq \frac{1}{4} d^2(M_{p-1}, N_0)+\frac{1}{2}d^2(M_{p-1}, N_{3})-\frac{3}{2p^2} d^{2}(A,C). \]
So we get
\begin{equation}\label{12}
d^2(M_{p-1},N_2)\geq \frac{1}{3} d^2(M_{p-1}, N_0)+\frac{2}{3}d^2(M_{p-1}, N_{3})-\frac{2}{ p^2} d^{2}(A,C).
\end{equation}
By substituting \eqref{12} into \eqref{3}, we obtain 
\[
\begin{split}
d^2(M_{p-1}, N_{3}) & \geq  \frac{1}{2}d^2(M_{p-1}, N_{2})+\frac{1}{2}d^2(M_{p-1}, N_{4})-\frac{d^2(A,C)}{p^2}\\
&\geq  \frac{1}{2}\left[\frac{1}{3} d^2(M_{p-1}, N_0)+\frac{2}{3}d^2(M_{p-1}, N_{3})-\frac{2}{ p^2} d^{2}(A,C) \right]+\frac{1}{2}d^2(M_{p-1}, N_{4})-\frac{d^2(A,C)}{p^2}\\
&= \frac{1}{6} d^2(M_{p-1}, N_0)+\frac{1}{3}d^2(M_{p-1}, N_{3})  +\frac{1}{2}d^2(M_{p-1}, N_{4})-\frac{2}{p^2} d^{2}(A,C).
\end{split} \]
This simplifies to
\[ \frac{2}{3} d^2(M_{p-1},N_3)\geq \frac{1}{6} d^2(M_{p-1}, N_0)+\frac{1}{2}d^2(M_{p-1}, N_{4})-\frac{2}{p^2} d^{2}(A,C). \]
Hence we conclude that
\begin{equation}\label{123}
d^2(M_{p-1},N_3)\geq \frac{1}{4} d^2(M_{p-1}, N_0)+\frac{3}{4}d^2(M_{p-1}, N_{4})-\frac{3}{ p^2} d^{2}(A,C).
\end{equation}
In view of  \eqref{12} and \eqref{123}, we now prove following inequalities by induction.
\begin{equation}
d^2(M_{p-1}, N_i) \geq \frac{1}{i+1}d^2(M_{p-1}, N_0) +\frac{i}{i+1}d^2(M_{p-1}, N_{i+1})-\frac{i}{ p^2} d^{2}(A,C),\quad j=1, 2, \dots, p-1.
\end{equation}
Assume that 
\[
d^2(M_{p-1}, N_i) \geq \frac{1}{i+1}d^2(M_{p-1}, N_0) +\frac{i}{i+1}d^2(M_{p-1}, N_{i+1})-\frac{i}{ p^2} d^{2}(A,C).
\]
By \eqref{4}, we have
\[
\begin{split}
d^2(M_{p-1}, N_{i+1})&\geq  \frac{1}{2}d^2(M_{p-1}, N_{i})+\frac{1}{2}d^2(M_{p-1}, N_{i+2})-\frac{d^2(A,C)}{p^2}\\
&\geq \frac{1}{2} \big( \frac{1}{i+1}d^2(M_{p-1}, N_0) +\frac{i}{i+1}d^2(M_{p-1}, N_{i+1})-\frac{i}{ p^2} d^{2}(A,C) \big)\\
&\quad +\frac{1}{2}d^2(M_{p-1}, N_{i+2})-\frac{d^2(A,C)}{p^2}\\
&=\frac{1}{2(i+1)} d^2(M_{p-1}, N_0)+\frac{i}{2(i+1)}d^2(M_{p-1}, N_{i+1})\\
&+\frac{1}{2} d^2(M_{p-1}, N_{i+2})-\frac{i+2}{ 2p^2} d^{2}(A,C),
\end{split}
\]
and then
\[ \frac{i+2}{2(i+1)}d^2(M_{p-1}, N_{i+1})\geq \frac{1}{2(i+1)} d^2(M_{p-1}, N_0) +\frac{1}{2} d^2(M_{p-1}, N_{i+2})-\frac{i+2}{ 2p^2} d^{2}(A,C). \]
So \[  d^2(M_{p-1}, N_{i+1})\geq \frac{1}{i+2} d^2(M_{p-1}, N_0) +\frac{i+1}{i+2} d^2(M_{p-1}, N_{i+2})-\frac{i+1}{  p^2} d^{2}(A,C). \]
Now we can conclude that
\begin{equation}
d^2(M_{p-1}, N_i) \geq \frac{1}{i+1}d^2(M_{p-1}, N_0) +\frac{i}{i+1}d^2(M_{p-1}, N_{i+1})-\frac{i}{ p^2} d^{2}(A,C), \quad i=1,2,...,p-1.
\end{equation}
 When $i=p-2$, we have
\begin{equation}\label{M2}
d^2(M_{p-1}, N_{p-2})\geq \frac{1}{p-1} d^2(M_{p-1}, N_{0})+\frac{p-2}{p-1} d^2(M_{p-1}, N_{p-1})-(p-2)\frac{d^2(A,C)}{p^2}
\end{equation}
Similarly, by  \eqref{N1}, for $i=1, 2, \dots, p-1$ we obtain
\begin{equation}
d^2(M_{i}, N_p) \geq \frac{1}{i+1}d^2(M_{0}, N_p) +\frac{i}{i+1}d^2(M_{i+1}, N_{p})-\frac{i}{ p^2} d^{2}(A,B).
\end{equation}
 In particular, taking $i=p-1$, we have
\begin{equation}\label{N2}
d^2(M_{p-1}, N_{p})\geq\frac{1}{p} d^2(M_{0}, N_{p})+ \frac{p-1}{p} d^2(M_{p}, N_{p})-(p-1)\frac{d^2(A,B)}{p^2}.
\end{equation}
 Note that  $d(M_{p-1},N_0)=\frac{p-1}{p} d(A,B)$, $d(M_0,N_p)=d(A,C)$ and $d(M_p,N_p)=d(B,C)$.
From \eqref{MN1}, \eqref{M2} and \eqref{N2}, we have
\[
\begin{split}
 d^2(M_{p-1}, N_{p-1})&\geq \frac{1}{2}d^2(M_{p-1}, N_{p-2}) +\frac{1}{2}d^2(M_{p-1}, N_{p})-\frac{d^2(A,C)}{p^2}\\
&\geq  \frac{1}{2}\big( \frac{1}{p-1} d^2(M_{p-1}, N_{0})+\frac{p-2}{p-1} d^2(M_{p-1}, N_{p-1})-(p-2)\frac{d^2(A,C)}{p^2} \big)\\
&+ \frac{1}{2}\big( \frac{1}{p} d^2(M_{0}, N_{p})+\frac{p-1}{p} d^2(M_{p}, N_{p})-(p-1)\frac{d^2(A,B)}{p^2} \big) -\frac{d^2(A,C)}{p^2}\\
&=\frac{p-1}{2p^2}d^{2}(A,B)+ \frac{p-2}{2(p-1)} d^2(M_{p-1},N_{p-1})-\frac{p-2}{2p^2}d^(A,C)+\frac{1}{2p}d^2(A,C)\\
&+\frac{p-1}{2p}d^2(B,C)-\frac{p-1}{2p^2} d^2(A,B)-\frac{1}{p^2}d^2(A,C)\\
&=\frac{p-2}{2(p-1)} d^2(M_{p-1},N_{p-1})+\big(\frac{p-1}{2p^2} -\frac{p-1}{2p^2}\big)d^{2}(A,B) \\
& +\big( \frac{1}{2p}-\frac{p-2}{2p^2}-\frac{1}{p^2}\big) d^2(A,C) +\frac{p-1}{2p}d^2(B,C) \\
&=\frac{p-2}{2(p-1)} d^2(M_{p-1},N_{p-1})+\frac{p-1}{2p}d^2(B,C).
\end{split}
\]
Therefore,  \[ \big(1- \frac{p-2}{2(p-1)}\big)d^2(M_{p-1},N_{p-1}) \geq \frac{p-1}{2p}d^2(B,C).\]
Simplifying, we get 
$d(M_{p-1},N_{p-1}) \geq \frac{(p-1)}{ p}d(B,C),
$
and hence
\[ 
\frac{p-1}{p} d(B,C) \leq d(A\sharp_{\frac{p-1}{p}} B, A\sharp_{\frac{p-1}{p}} C), \quad p=1, 2, \dots 
\]
\end{proof}

Next, we extend Lemma \ref{forp} to show that $t d(B,C) \leq d(A\sharp_t B, A\sharp_t C)$ for any $t\in [0,1].$

\begin{theorem}\label{fort2}
Let $A, B, C\in  \Gr_{n,k}(\F)$ lie within a metric ball $\B$ of radius less than $\rho_\F$. Then
\begin{equation}\label{anyt2}
 d(A\sharp_t B, A\sharp_t C) \geq t d(B,C), \quad t\in [0,1].
\end{equation}
\vspace{-1cm}
\begin{figure}[H]
\begin{center}
\includegraphics[width=0.4\textwidth, angle=0]{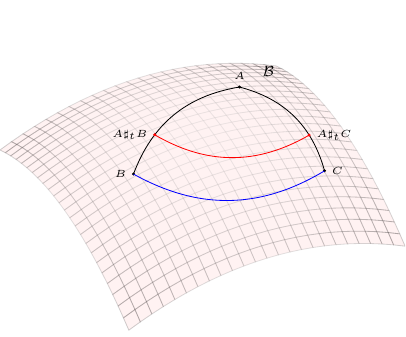}
\end{center}
\caption{$ d(A\sharp_t B, A\sharp_t C) \geq t d(B,C), \quad t\in [0,1]$.}
\end{figure}
\end{theorem}

\begin{proof}
Let $k$ and $p$ be positive integers with $k\leq p$. Let 
\[
B_0=A\sharp_{\frac{k+1}{p}}B,\quad C_0=A\sharp_{\frac{k+1}{p}}C.
\]
By Lemma \ref{forp}, we have
\[
 d(A_{\frac{k}{k+1}}B_0, A\sharp_{\frac{k}{k+1}} C_0) \geq \frac{k}{k+1} d(B_0,C_0). 
\]
Using \eqref{st}, we obtain
\[
A\sharp_{\frac{k}{k+1}}B_0=A\sharp_{\frac{k}{k+1}}\big( A\sharp_{\frac{k+1}{p}}B\big)=A\sharp_{\frac{k}{p}}B,
\quad
A\sharp_{\frac{k}{k+1}}C_0=A\sharp_{\frac{k}{k+1}}\big( A\sharp_{\frac{k+1}{p}}C\big)=A\sharp_{\frac{k}{p}}C.
\]
They lead to
 \[
 d(A\sharp_{\frac{k}{p}}B, A\sharp_\frac{k}{p} C) \geq \frac{k}{k+1} d(A_{\frac{k+1}{p}}B, A\sharp_\frac{k+1}{p} C) \quad \mbox{for all } k\leq q.
 \]
Inductively,
\[ \begin{split}
d(A\sharp_{\frac{k}{p}}B, A\sharp_\frac{k}{p} C)&\geq \frac{k}{k+1} d(A_{\frac{k+1}{p}}B, A\sharp_\frac{k+1}{p} C)\\
&\geq  \frac{k}{k+1} \frac{k+1}{k+2}d(A_{\frac{k+2}{p}}B, A\sharp_\frac{k+2}{p} C)\\
& \ \vdots\\
&\geq  \frac{k}{k+1} \frac{k+1}{k+2} \cdots \frac{p-1}{p}  d(B,C)=\frac{k}{p} d(B,C).
\end{split}
 \]
Thus $d(A\sharp_t B, A\sharp_t C) \geq t d(B,C)$ holds for any rational number  $t\in [0,1]$. 
Finally, by continuity, we conclude that $d(A\sharp_t B, A\sharp_t C) \geq t d(B,C)$ for any $t\in [0,1].$ 
\end{proof}

\begin{remark}\label{remark_nonnegative_curvature}
Theorem \ref{fort2} also holds for any non-negative curvature manifold $\mathcal{M}$ within a locally convex ball $\mathcal{B}$, where the geodesic between any two points in $\mathcal{B}$ is unique. Specifically, if $A, B, C \in \mathcal{M}$ are contained within such a ball $\mathcal{B}$ of radius less than $\rho_\mathcal{M}$, then for any $t \in [0,1]$, the inequality
\[
d(A\sharp_t B, A\sharp_t C) \geq t d(B,C)
\]
holds.
Moreover, the technique we use to prove the generalized inequality \eqref{anyt2} in this section can be adapted to the analogous inequality $d(A\sharp_t B, A\sharp_t C) \leq t d(B,C)$ for non-positive curvature manifolds, with the main difference being that the inequality holds in the opposite direction and globally.
\end{remark}

For the hyperbolic $\P_n$, the distance function $d_{\P_n}(\cdot, \cdot)$ in \eqref{d_P} is convex globally \cite [Corollary 3.11]{LLT14} on $[0,1]$, i.e.,
\begin{equation}
d_{\P_n}(B_1\sharp B_2, C_1\sharp C_2) \leq (1-t)d_{\P_n}(B_1,C_1)+t d_{\P_n}(B_2,C_2), \quad t\in [0,1],
\end{equation}
holds globally. In contrast, the numerical examples below tell us that the distance function $d(\cdot, \cdot)$ for the Grassmannians is neither convex nor concave globally.

\begin{example}\label{3.5}
Consider the following four matrices in $\Gr_{2,2}(\R)$:
\[B_1=\begin{pmatrix} 0.9414 & 0.2348\\ 0.2348 & 0.0586 \end{pmatrix} \quad B_2=\begin{pmatrix}  0.9998 & 0.0144\\ 0.0144 & 0.0002 \end{pmatrix}\]
\[C_1=\begin{pmatrix}   0.9969  &  -0.0560\\-0.0560 &0.0031\end{pmatrix} \quad
 C_2=\begin{pmatrix}  0.1533& 0.3603\\ 0.3603 &0.8467 \end{pmatrix}.\]
The unitary matrices $(I-2B_1)(I-2B_2)$, $(I-2C_1)(I-2C_2)$, $(I-2B_1)(I-2C_1)$, $(I-2B_2)(I-2C_2)$ have no
negative eigenvalues.
From \eqref{dist},  the distances are $d(B_1,C_1)=1.6321$,
$d(B_2,C_2)=0.4250$ and $d(B_1\sharp B_2, C_1\sharp C_2)=0.6035$.
Then $$d(B_1\sharp B_2, C_1\sharp C_2) <\frac{1}{2}d(B_1,C_1)+\frac{1}{2}d(B_2,C_2).$$\end{example}

\begin{example}\label{3.6}
Consider the following four matrices in $\Gr_{2,2}(\R)$:
\[B_1=\begin{pmatrix}
0.2793  &  0.4486\\
    0.4486  &  0.7207
 \end{pmatrix} \quad B_2=\begin{pmatrix}  
0.4345 &  -0.4957\\
   -0.4957   & 0.5655
\end{pmatrix}\]
\[C_1=\begin{pmatrix}  
0.7691  &  0.4214\\
    0.4214  &  0.2309
\end{pmatrix} \quad
 C_2=\begin{pmatrix}  0.6708 &  -0.4699\\
   -0.4699  &  0.3292
\end{pmatrix}.\]
The unitary matrices $(I-2B_1)(I-2B_2)$, $(I-2C_1)(I-2C_2)$, $(I-2B_1)(I-2C_1)$, $(I-2B_2)(I-2C_2)$ have no 
negative eigenvalues.
From \eqref{dist},  the distances are $d(B_1,C_1)=0.7251$,
$d(B_2,C_2)= 0.3395$ and $d(B_1\sharp B_2, C_1\sharp C_2)=1.8589$.
Then $$d(B_1\sharp B_2, C_1\sharp C_2) > \frac{1}{2} d(B_1,C_1)+\frac{1}{2}d(B_2,C_2).$$\end{example}
 
Note that we are not sure if $d(B_1\sharp B_2, C_1\sharp C_2) \geq (1-t)d(B_1,C_1)+t d(B_2,C_2)$ holds for  $t\in [0,1]$ when $B_1$, $B_2$, $C_1$ and $C_2$ are within the metric ball $\B$ as it would involve $\rho_\F$. 
\begin{remark}
The reverse versions of inequalities (\ref{midineq}) and (\ref{anyt2}) for $\P_n$  can be found in \cite{RB07}, reflecting the hyperbolic geometry of $\P_n$.
These results were further extended to symmetric spaces of noncompact type  \cite [Theorem 3.10]{LLT14}.
\end{remark}

\section{Conclusion}

In this paper, we provided the geometric properties of Grassmannians, focusing on distance inequalities and geodesic structures. We first established the Semi-Parallelogram Law and the Law of Cosines, providing foundational insights into the nature of geodesic triangles in Grassmannians. We then proved that the inequality $d(A \sharp B, A \sharp C) \geq \frac{1}{2} d(A, B)$ holds for geodesic triangles with sufficiently close vertices $A$, $B$, and $C$.
Moreover, we extended this result to a more general setting by proving that for any $t \in [0, 1]$, the inequality $ d(A\sharp_t B, A\sharp_t C) \geq t d(B,C)$ holds. 
this work lays a foundation for further exploration of Grassmannian structures and their applications in various mathematical and computational domains.

\vspace{.5cm}
\noindent{\bf Declaration of competing interest}\\
The authors declare that they have no known competing financial interests or personal relationships that could have appeared to influence the work reported in this paper.

\vspace{.5cm}
\noindent{\bf Data availability}\\
No data was used for the research described in the article.

\bibliographystyle{abbrv}
\bibliography{refs}

\end{document}